\documentclass[a4paper,oneside,11pt]{article}

\usepackage{amsmath,amsfonts,amscd,amssymb}
\usepackage{longtable,geometry}
\usepackage[english]{babel}
\usepackage[utf8]{inputenc}
\usepackage[active]{srcltx}
\usepackage[T1]{fontenc}
\usepackage{graphicx}
\usepackage{pstricks}
\usepackage{bbm}
\usepackage{MnSymbol}

\newtheorem{theorem}{Theorem}[section]

\newtheorem{corollary}[theorem]{Corollary}
\newtheorem{lemma}[theorem]{Lemma}
\newtheorem{proposition}[theorem]{Proposition}
\newtheorem{definition}[theorem]{Definition}
\newtheorem{example}[theorem]{Example}

\newtheorem{remark}[theorem]{Remark}

\newcommand{\be}[1]{\begin{equation}\label{#1}}
\newcommand{\ee}{\end{equation}}
\numberwithin{equation}{section}

\newenvironment{proof}[1][\relax]%s
  {\paragraph{Proof\ifx#1\relax\else~of #1\fi}}%
  {~\hfill$\square$\par\bigskip}
  {\paragraph{Proof\ifx#1\relax\else~de #1\fi}}%
  {~\hfill$\square$\par\bigskip}

\newcommand{\calH}{\mathcal{H}}
\newcommand{\calI}{\mathcal{I}}

\newcommand{\bbC}{\mathbb{C}}

\newcommand{\bbN}{\mathbb{N}}

\newcommand{\bbR}{\mathbb{R}}

\newcommand{\bbZ}{\mathbb{Z}}
\renewcommand{\be}{\beta}

\renewcommand{\Re}{\textrm{Re}}

\newcommand{\rk}[1]{\bgroup\color{red}%
  \par\medskip\hrule\smallskip%
  \noindent\textbf{#1}%
  \par\smallskip\hrule\medskip\egroup}

\title{The bundle Laplacian on discrete tori}
\author{Fabien Friedli%
	\thanks{The author was supported in part by the Swiss NSF grant 200021 132528/1.%
	}}

\begin{document}
\maketitle

\begin{abstract}
We prove an asymptotic formula for the determinant of the bundle Laplacian on discrete $d$-dimensional tori as the number of vertices tends to infinity. This determinant has a combinatorial interpretation in terms of cycle-rooted spanning forests. We also establish a relation (in the limit) between the spectral zeta function of a line bundle over a discrete torus, the spectral zeta function of the infinite graph $\bbZ^d$ and the Epstein-Hurwitz zeta function. The latter can be viewed as the spectral zeta function of the twisted continuous torus which is the limit of the sequence of discrete tori.
\end{abstract}


\begin{section}{Introduction}
	
The number of spanning trees in a graph is an important quantity in combinatorics, probability, statistical physics and other fields, and has been studied extensively. The main tool used to count spanning trees is the matrix-tree theorem by Kirchhoff, which relates their number to the determinant of the combinatorial Laplacian. Thus this combinatorial problem can be translated into a spectral one. In \cite{Ken11} Kenyon develops the theory of the \emph{vector bundle Laplacian}, first studied by Forman \cite{For93}, in order, among other things, to obtain results on the loop-erased random walk on lattices (see also his paper \cite{Ken15}). There is an analog of the matrix-tree theorem in this setting, relating the determinant of the bundle Laplacian to cycle-rooted spanning forests, see Section 2.\\
In statistical physics in particular, it is often interesting to look at sequences of graphs whose number of vertices go to infinity and to relate the combinatorics of such sequences to continuous objects in the limit. If the graphs are discrete tori and we are interested in the number of spanning trees, this was carried out, in all dimensions, by Chinta, Jorgenson and Karlsson in \cite{CJK10}. They show in particular that the constant term in the asymptotics is the regularized determinant of the continuous torus. We refer to \cite{Res15} for an explanation of what is expected to hold in general (for other graphs) and relations with quantum field theory.\\
In the present paper we use the ideas of \cite{CJK10} to establish an asymptotic formula for the determinant of the bundle Laplacian on discrete tori when the number of vertices goes to infinity.\\
Let $G_n$ be the line bundle over the discrete torus defined in Section 2 and write $\Delta$ for the bundle Laplacian on $G_n$. Suppose that $\lambda_i\in[0,1]$ for each $i\in\{1,\ldots,d\}$ and that $\lambda_i\notin \{0,1\}$ for some index $i$. Our main result is
\begin{theorem}\label{thm11}
For an integer $d\geq 1$ write $$c_d=-\int_0^\infty\Big(e^{-2dt}I_0(2t)^d-e^{-t}\Big)\frac{dt}{t}$$ and, for $\Re(s)>\frac{d}{2}$, $$\zeta_{EH}(s;\alpha_1,\ldots,\alpha_d;\lambda_1,\ldots,\lambda_d)=(2\pi)^{-2s}\sum_{K\in\bbZ^d}\Big(\Big(\frac{k_1+\lambda_1}{\alpha_1}\Big)^2+\ldots+\Big(\frac{k_d+\lambda_d}{\alpha_d}\Big)^2\Big)^{-s}.$$ Then, as $n\longrightarrow\infty$, $$\log\det\Delta=\Big(\prod_{i=1}^{d}a_i(n)\Big)c_d-\zeta_{EH}'(0;\alpha_1,\ldots,\alpha_d;\lambda_1,\ldots,\lambda_d)+o(1).$$
\end{theorem}
The constant $c_d$ is the same as for the spanning trees and it is known that $c_1=0$ and $c_2=\frac{4G}{\pi}$, where $G$ is Catalan's constant (see \cite{CJK10} or \cite{SS01}). The difference lies in the second term (if we forget about the $\log(u^2)$ term in \cite{CJK10}).\\
In dimension $d=2$ there is a nice expression for $\zeta_{EH}'(0;\alpha_1,\alpha_2;\lambda_1,\lambda_2)$ in the spirit of the famous Kroncecker limit formula, which adds some interest to this asymptotics independently of the combinatorial setting.
\begin{theorem}\label{Kronecker}
If $d=2$ we have $$\zeta_{EH}'(0;\alpha_1,\alpha_2;\lambda_1,\lambda_2)=2\pi\frac{\alpha_1}{\alpha_2}B_2(\lambda_2)-2\log\prod_{n\in\bbZ}\Big|1-e^{2\pi i\lambda_1}e^{-2\pi\frac{\alpha_1}{\alpha_2}|n+\lambda_2|}\Big|,$$ where $B_2(x)=x^2-x+\frac{1}{6}$ stands for the second Bernoulli polynomial.
\end{theorem}
After giving some definitions in Section 2, we compute the heat kernel on $G_n$ and determine asymptotics for the associated theta functions in Section 3. The proof of Theorem \ref{thm11} is carried out in Section 4.
\\
Finally, in Section 5 we consider the spectral zeta function of $G_n$ as defined in \cite{FK16} and show the following.
\begin{theorem}\label{zeta}
Let $s\in\bbC$ such that $s\neq m +\frac{d}{2}$ for $m\in\bbN$. Then, as $n\longrightarrow\infty$, we have $$\zeta_{G_n}(s)=\Big(\prod_{i=1}^{d}a_i(n)\Big)\zeta_{\bbZ^d}(s)+\zeta_{EH}(s;\alpha_1,\ldots,\alpha_d;\lambda_1,\ldots,\lambda_d)n^{2s}+o(n^{2s}).$$
\end{theorem}
This should be compared with the results obtained in \cite{FK16} and \cite{Fri16} where similar formulas were obtained in the case of the standard Laplacian and spectral $L$-functions, respectively.\\

In this work we only consider diagonal discrete tori, but with a bit of work, more or less the same results should hold for arbitrary discrete tori, see \cite{CJK12}.\\

\noindent
\textbf{Acknowledgement. }The author is grateful to Anders Karlsson for suggesting this problem to him and for useful discussions and comments on this project. The author also thanks Justine Louis and Pham Anh Minh for interesting discussions related to this work.  
\end{section}

\begin{section}{Definitions}

As explained by Kenyon in \cite{Ken11}, given a finite graph $G$, we can construct a vector bundle over it. To each vertex of $G$ we associate a vector space isomorphic to a given vector space. For each oriented edge we can then choose an isomorphism between the two vector spaces attached to the end-points of that edge, with the condition that the isomorphism corresponding to an oriented edge must be the inverse of the isomorphism corresponding to the same edge oriented in the opposite direction. The set of these isomorphisms is called the \emph{connection} of $G$. We have a natural generalization of the Laplacian operator on such graphs, the \emph{bundle Laplacian}. It acts on the functions $f:VG\longrightarrow\bbC$ (where $VG$ denotes the set of vertices of $G$) and it is defined by
\begin{displaymath}
	\Delta f(v)=\sum_{w\sim v}(f(v)-\phi_{w,v}f(w)),
\end{displaymath}
where the sum is over all adjacent vertices and $\phi_{w,v}$ denotes the isomorphism for the oriented edge $wv$.
\begin{remark}
\emph{If the bundle is trivial (all the isomorphisms are identity), this is the usual Laplacian.}
\end{remark}
For the standard Laplacian we know by the matrix-tree theorem of Kirchhoff that its determinant (in fact the product of the non-zero eigenvalues) counts spanning trees in the graph. We have a similar combinatorial interpretation here. For this we only consider \emph{line bundles}, that is bundles of dimension one. In this case we associate a copy of $\bbC$ to each vertex and the isomorphisms are just multiplication by a non-zero complex number.  We can see this process as a choice of a complex weight on each oriented edge (with the inverse weight for the same edge with opposite orientation), but the bundle Laplacian should not be confused with what is usually called the weighted Laplacian (see for example \cite{Moh91}).
\\
Given an oriented cycle, the product of the weights on the oriented edges along the cycle is called the \emph{monodromy} of the cycle.
\\
A subset of the set of the edges of a given graph which spans all the vertices of the graph and such that each connected component has exactly one cycle is called a \emph{cycle-rooted spanning forest} and abreviated $CRSF$. The analog of Kirchhoff theorem is then (\cite{Ken11})
\begin{theorem}\label{Kenyon}
	For a line bundle on a connected finite graph,
	\begin{displaymath}
		\det\Delta=\sum_{CRSF's}\prod_{cycles}(2-w-\frac{1}{w}),
	\end{displaymath}
	where the sum is over all unoriented CRSF's $C$, the product is over cycles of $C$ and $w,\frac{1}{w}$ are the monodromies of the two orientations of the cycle.
\end{theorem}

\begin{remark}
\emph{If the weights on the edges are chosen to be of modulus one, the bundle is called \emph{unitary} and the bundle Laplacian becomes Hermitian and positive semidefinite.}
\end{remark}
We will evaluate $\det\Delta$ for a sequence of line bundles over discrete tori, when the number of vertices goes to infinity. More precisely, let $d\geq 1$ be an integer (the "dimension"). For each $i\in\{1,\ldots, d\}$, let $a_i(n)$ be a sequence of integers such that $$\lim_{n\rightarrow\infty}\frac{a_i(n)}{n}=\alpha_i>0.$$ For every $n$ and $i\in\{1,\ldots,d\}$ we associate a complex number $w_{i,j}(n)$ of modulus one to the oriented edge between $j$ and $j+1$ (with $0\leq j\leq a_i(n)-1$) in the Cayley graph of $\bbZ/a_i(n)\bbZ$ with generators $\{\pm 1\}$.  We consider the \emph{discrete torus} defined by the Cayley graph of $$\prod_{i=1}^d\bbZ/a_i(n)\bbZ$$ (with generators given by $(0,\ldots,0,\pm1,0,\ldots,0)$) and the natural line bundle which comes with it (that is, the weight of an oriented edge in this graph is given by the weight associated to the corresponding edge in some $Cay(\bbZ/a_i(n)\bbZ)$). We denote this graph (with the line bundle) by $G_n$. Note that this graph depends on several parameters, namely $d$, $n$, $a_i(n)$ and $w_{i,j}(n)$, but in order to simplify the notation we only write the dependence in $n$.
\end{section}
\begin{section}{Heat kernel and theta functions}
We adapt the method used in \cite{CJK10} and \cite{FK16} to compute asymptotics for $\det\Delta$ and for the spectral zeta function associated to $G_n$ for the sequence of discrete tori described above. The first step is to compute the heat kernel of the graph $G_n$, that is the unique bounded solution $$K:\bbR_{+}\times G_n\longrightarrow\bbR$$ of the equation
\begin{equation*}
(\Delta+\frac{\partial}{\partial t})K(t,x)=0
\end{equation*}
with initial condition $K(0,x)=\delta_0(x)$, where $\delta$ is the Kronecker delta and $0$ means the vertex corresponding to $(0,\ldots,0)$ in $G_n$.\\
The existence and uniqueness of such a function is established for a general class of graphs and for the standard Laplacian in \cite{DM06} and \cite{Dod06}. Here we do not need a general theory, because it is possible and quite easy to check the uniqueness of the solution found in Proposition \ref{heatkernelformula} by taking the Fourier transform and solving the corresponding differential equation.
\begin{proposition}\label{heatkernelformula}
The heat kernel for the graph $G_n$ defined above is given by $$K(t,x)=e^{-2dt}\sum_{K\in\bbZ^d}\prod_{i=1}^{d}I_{x_i+k_ia_i(n)}(2t)\prod_{j=0}^{a_i(n)-1}w_{i,j}(n)^{[\frac{j-x_i}{a_i(n)}]-k_i},$$ where we write $K=(k_1,\ldots,k_d)$, $x=(x_1,\ldots,x_d)\in\prod_{i=1}^d\bbZ/a_i(n)\bbZ$ and $I_x$ for the modified Bessel function of the first kind of order $x$.
\end{proposition}
\begin{proof}
First note that $K(t,x)$ is well-defined in the sense that we have $K(t,x)=K(t,y)$ if $x\equiv y$ in $\prod_{i=1}^d\bbZ/a_i(n)\bbZ$. Also the infinite sum is convergent and is bounded in $t$, as can be seen using the series representation for the modified Bessel function $I_x$, see \cite{KN06}.\\
\noindent
For $i\in\{1,\ldots,d\}$ and $x_i\in\bbZ/a_i(n)\bbZ$ define $$K_i(t,x_i):=e^{-2t}\sum_{k\in\bbZ}I_{x_i+ka_i(n)}(2t)\prod_{j=0}^{a_i(n)-1}w_{i,j}(n)^{[\frac{j-(x_i+ka_i(n))}{a_i(n)}]}.$$ With this notation we have $K(t,x)=\prod_{i=1}^{d}K_i(t,x_i)$, where $x=(x_1,\ldots,x_d)$.\\

Let $x_i\in\bbZ/a_i(n)\bbZ$ and write $x_i=m+ra_i(n)$ with $0\leq m \leq a_i(n)-1$ (for convenience we do not write the dependence in $x_i$ for $m$ and for $r$). Since $\Big[\frac{j-(m+1)}{a_i(n)}\Big]=\Big[\frac{j-m}{a_i(n)}\Big]$ if $j\neq m$ and $\Big[\frac{j-(m+1)}{a_i(n)}\Big]=\Big[\frac{j-m}{a_i(n)}\Big]-1$ if $j=m$ we observe that
\begin{align*}
K_i(t,x_i+1)&=e^{-2t}\sum_{k\in\bbZ}I_{x_i+1+ka_i(n)}(2t)\prod_{j=0}^{a_i(n)-1}w_{i,j}^{k+r+[\frac{j-(m+1)}{a_i(n)}]}\\&=w_{i,m}e^{-2t}\sum_{k\in\bbZ}I_{x_i+1+ka_i(n)}(2t)\prod_{j=0}^{a_i(n)-1}w_{i,j}^{[\frac{j-(x_i+ka_i(n))}{a_i(n)}]}.
\end{align*}
Similarly we have $$K_i(t,x_i-1)=w_{i,m-1}^{-1}e^{-2t}\sum_{k\in\bbZ}I_{x_i-1+ka_i(n)}(2t)\prod_{j=0}^{a_i(n)-1}w_{i,j}^{[\frac{j-(x_i+ka_i(n))}{a_i(n)}]},$$ where $m-1$ must be understood modulo $a_i(n)$, that is if $m=0$ the weigth above is in fact $w_{i,a_i(n)-1}^{-1}$. We keep this convention for the rest of the proof.
Therefore, using the relation $I_x'(2t)=I_{x-1}(2t)+I_{x+1}(2t)$, we have
\begin{equation*}
-\frac{\partial}{\partial t}K_i(t,x_i)=2K_i(t,x_i)-w_{i,m}^{-1}K_i(t,x_i+1)-w_{i,m-1}K_i(t,x_i-1).
\end{equation*}
In other words, $K_i$ is a solution to the heat equation on the $i$-th copy of our Cayley graph $\prod_{i=1}^{d}\bbZ/a_i(n)\bbZ.$ Now we can compute
\begin{align*}
-\frac{\partial}{\partial t}K(t,x)&=\sum_{i=1}^{d}(-\frac{\partial}{\partial t}K_i(t,x_i))\prod_{l\neq i}K_l(t,x_l)\\&=2dK(t,x)-\sum_{i=1}^{d}(w_{i,m}^{-1}K_i(t,x_i+1)+w_{i,m-1}K_i(t,x_i-1))\prod_{l\neq i}K_l(t,x_l)\\&=\Delta K(t,x).
\end{align*}
Using the fact that $I_0(0)=1$ and $I_m(0)=0$ for all $m\in\bbZ^*$ it is easy to check that $K(t,x)=\delta_0(x)$, which completes the proof.
\end{proof}
\begin{remark}
\emph{Notice that in the proof we show in fact that the heat kernel of the product graph is equal to the product of the heat kernels on each cyclic copy. To guess the formula on one copy we computed the heat kernel on $\bbZ$ with periodic weights and then took the quotient by making the function we obtained periodic. See \cite{KN06} for more details about this procedure in the standard case without the weights.}
\end{remark}
From now on, we write $K_i(t)$ for $K_i(t,0)$ and $K(t)$ for $K(t,0)$. As explained in \cite{FK16} the spectral zeta function of a finite graph is the Mellin transform of the trace of the heat kernel or, equivalently, of the theta function associated to the Laplacian acting on this graph. Here the theta function associated to $G_n$ is given by $$\theta^{G_n}(t):=\prod_{i=1}^{d}a_i(n)K_i(t),$$ where $K_i$ is as in the proof above, that is $$K_i(t)=e^{-2t}\sum_{k\in\bbZ}I_{ka_i(n)}(2t)\prod_{j=0}^{a_i(n)-1}w_{i,j}(n)^{-k}.$$
For convenience we write $\theta_i^{G_n}(t)=a_i(n)K_i(t)$. From this last expression, it is easy to see that $\theta_i^{G_n}(t)\sim a_i(n)e^{-2t}I_0(2t)$ when $t\longrightarrow 0$ (see Lemma \ref{asymp}). We will also need to know the behavior of $\theta_i^{G_n}(t)$ when $t\longrightarrow\infty$.
\begin{proposition}\label{theta}
Let $\lambda_i(n)\in[0,1)$ such that $\prod_{j=0}^{a_i(n)-1}w_{i,j}=e^{2\pi i \lambda_i(n)}$. For any $t\in\bbR$ we have $$\theta_i^{G_n}(t)=\sum_{j=0}^{a_i(n)-1}e^{-4t\sin^2\Big(\frac{\pi(j+\lambda_i(n))}{a_i(n)}\Big)}.$$
\end{proposition}
For the proof we need a formula about Bessel functions that we could not find explicitly in the literature.
\begin{lemma}\label{Besselgen}
For any $t\in\bbC^*$, $z\in\bbC$ and $n\geq 1$ we have $$\sum_{k\in\bbZ}t^{kn}I_{kn}(z)=\frac{1}{n}\sum_{j=0}^{n-1}\exp\Big(\frac{z}{2}(t^{-1}e^{\frac{2\pi i j}{n}}+te^{\frac{-2\pi i j}{n}})\Big).$$
\end{lemma}
\begin{proof}
Consider the function of the real variable $x$ defined by $e^{\frac{z}{2}(\frac{1}{t}e^{ix}+te^{-ix})}$. It is $2\pi$-periodic in $x$ and derivable so by Fourier analysis we can write $$e^{\frac{z}{2}(\frac{1}{t}e^{ix}+te^{-ix})}=\sum_{k\in\bbZ}L_k(z,t)e^{ikx},$$ where $L_k(z,t)=\frac{1}{2\pi}\int_{-\pi}^{\pi}e^{\frac{z}{2}(\frac{1}{t}e^{i\tau}+te^{-i\tau})}e^{-ik\tau}d\tau.$ If we substitute $x=\frac{2\pi j}{n}$ and sum over $j=0,\ldots,n-1$ we obtain $$\sum_{j=0}^{n-1}e^{\frac{z}{2}(\frac{1}{t}e^{\frac{2\pi i j}{n}}+te^{\frac{-2\pi i j}{n}})}=n\sum_{k\in\bbZ}L_{kn}(z,t).$$ Hence it only remains to show that $\sum_{k\in\bbZ}L_{kn}(z,t)=\sum_{k\in\bbZ}t^{kn}I_{kn}(z)$. In order to do this write $$F_m(z,t)=\sum_{k\in\bbZ \atop k\equiv m (n)}L_k(z,t)$$ and $$G_m(z,t)=\sum_{k\in\bbZ \atop k\equiv m (n)}t^kI_k(z).$$ We want to prove that $F_0(z,t)=G_0(z,t)$. From the definition of $L_{k}$ we observe that $$\frac{d}{dz}L_{k}(z,t)=\frac{1}{2}\Big(\frac{1}{t}L_{k-1}(z,t)+tL_{k+1}(z,t)\Big)$$ and so we have $$\frac{d}{dz}F_m(z,t)=\frac{1}{2}\Big(\frac{1}{t}F_{m-1}(z,t)+tF_{m+1}(z,t)\Big),$$ for all $m\in\{0,\ldots,n-1\}.$ We are left with a simple system of linear differential equations with matrix $A=Circ(0,\frac{t}{2},0,\ldots,0,\frac{1}{2t})$, where $Circ(v)$ means a circulant matrix with vector $v$, that is $$\frac{d}{dz}\left( \begin{array}{c}
F_0(z,t) \\
\vdots \\
F_{n-1}(z,t)
\end{array} \right)=A\left( \begin{array}{c}
F_0(z,t) \\
\vdots \\
F_{n-1}(z,t)
\end{array} \right).$$ 

The solution is given by the vector $$e^{zA}\left( \begin{array}{c}
F_0(0,t) \\
\vdots \\
F_{n-1}(0,t)
\end{array} \right).$$
It is obvious from the definition that $L_0(0,t)=1$  and $L_k(0,t)=0$ if $k\neq 0$. Therefore we have $F_0(z,t)=c_{1,1}(e^{zA})$, where $c_{1,1}$ stands for the upper left entry of the matrix.\\

From classical properties of modified Bessel functions, we deduce that a similar system is satisfied by the functions $G_m(z,t)$, namely $$\frac{d}{dz}G_m(z,t)=\frac{1}{2}\Big(tG_{m-1}(z,t)+\frac{1}{t}G_{m+1}(z,t)\Big)$$ for all $m\in\{0,\ldots,n-1\}$. We note that the associated matrix here is $A^T$. Thus we have $G_0(z,t)=c_{1,1}(e^{zA^T})=c_{1,1}((e^{zA})^T)=c_{1,1}(e^{zA})=F_0(z,t)$, since the initial conditions are the same. This completes the proof.
\end{proof}
Now we can easily prove Proposition \ref{theta}.
\begin{proof}[Proposition \ref{theta}]
In view of Lemma \ref{Besselgen} we have
\begin{align*}
\theta_i^{G_n}(t)&=a_i(n)e^{-2t}\sum_{k\in\bbZ}I_{ka_i(n)}(2t)\prod_{j=0}^{a_i(n)-1}w_{i,j}(n)^{-k}\\&=a_i(n)e^{-2t}\sum_{k\in\bbZ}I_{ka_i(n)}(2t)\Big(e^{-\frac{2\pi i\lambda_i(n)}{a_i(n)}}\Big)^{ka_i(n)}\\&=e^{-2t}\sum_{j=0}^{a_i(n)-1}\exp\Big(t\Big(e^{\frac{2\pi i \lambda_i(n)}{a_i(n)}}e^\frac{2\pi i j}{a_i(n)}+e^{-\frac{2\pi i\lambda_i(n)}{a_i(n)}}e^{-\frac{2\pi i j}{a_i(n)}}\Big)\Big)\\&=\sum_{j=0}^{a_i(n)-1}e^{-4t\sin^2\Big(\frac{\pi(j+\lambda_i(n))}{a_i(n)}\Big)}.
\end{align*}
\end{proof}
\begin{remark}\label{eigen}
\emph{As a consequence of this formula we have that the Laplace eigenvalues of $G_n$ are given by the set $$\{4\sin^2\Big(\frac{\pi(j_1+\lambda_1(n))}{a_1(n)}\Big)+\ldots+4\sin^2\Big(\frac{\pi(j_d+\lambda_d(n))}{a_d(n)}\Big), 0\leq j_m\leq a_i(n)-1\}.$$ Indeed, our conditions on the weights ensure that the bundle Laplacian we consider here is Hermitian and positive definite, so the same reasoning as in \cite{KN06} on p.180 is valid. This shows that the heat kernel at $0$ is in fact the trace of $e^{-t\Delta}$.}
\end{remark}
For each $1\leq i\leq d$ fixed, we will assume without loss of generality that the sequence $\lambda_i(n)$ converges, taking a subsequence if necessary, by compactness. We write $$\lambda_i:=\lim\limits_{n\rightarrow\infty}\lambda_i(n)\in [0,1].$$
For $t>0$, we define $$\theta_i^{\infty}(t)=\sum_{k\in\bbZ}e^{-4(\frac{\pi}{\alpha_i})^2t(k+\lambda_i)^2}$$ and $$\theta^{\infty}(t):=\prod_{i=1}^{d}\theta_i^\infty(t)=\sum_{K\in\bbZ^d}e^{-4\pi^2 t\sum_{i=1}^{d}\Big(\frac{k_i+\lambda_i}{\alpha_i}\Big)^2}.$$
\begin{lemma}\label{poisson}
For all $t>0$ we have $$\theta_i^{\infty}(t)=\frac{\alpha_i}{\sqrt{4\pi t}}\sum_{k\in\bbZ}e^{-\frac{(\alpha_i k)^2}{4t}-2\pi i\lambda_i k}.$$
\end{lemma}
\begin{proof}
The sum in the right-hand side can be written as
\begin{align*}
\sum_{k\in\bbZ}e^{-\frac{(\alpha_i k)^2}{4t}-2\pi i\lambda_i k}&=e^{-4(\frac{\pi \lambda_i}{\alpha_i})^2t}\sum_{k\in\bbZ}e^{-(\frac{\alpha_i k}{\sqrt{4t}}+i\frac{\sqrt{4t}\pi\lambda_i}{\alpha_i})^2}\\&=e^{-4(\frac{\pi \lambda_i}{\alpha_i})^2t}\sum_{k\in\bbZ}f(k+i\frac{4t\pi\lambda_i}{\alpha_i^2}),
\end{align*}
where $f(y):=e^{-\frac{(\alpha_i y)^2}{4t}}$. This function has a simple Fourier transform, namely $\hat{f}(\nu)=\frac{\sqrt{4\pi t}}{\alpha_i}e^{-4(\frac{\pi\nu}{\alpha_i})^2t}$. By Poisson summation formula we conclude that
\begin{align*}
e^{-4(\frac{\pi \lambda_i}{\alpha_i})^2t}\sum_{k\in\bbZ}f(k+i\frac{4t\pi\lambda_i}{\alpha_i^2})&=e^{-4(\frac{\pi \lambda_i}{\alpha_i})^2t}\sum_{k\in\bbZ}\frac{\sqrt{4\pi t}}{\alpha_i}e^{-4(\frac{\pi k}{\alpha_i})^2t}e^{2\pi i(i\frac{4t\pi\lambda_i}{\alpha_i^2})k}\\&=\frac{\sqrt{4\pi t}}{\alpha_i}\sum_{k\in\bbZ}e^{-4(\frac{\pi}{\alpha_i})^2t(k+\lambda_i)^2}.
\end{align*}
\end{proof}
We consider the case where the bundle does not become trivial asymptotically, that is we suppose that there exists $i\in\{1,\ldots,d\}$ such that $\lambda_i\notin\{0,1\}$. Taking $n$ big enough, we can always assume that, for this index $i$, we have $\lambda_i(n)\neq 0$ for every $n$.
\begin{lemma}\label{asymp}
The following asymptotics hold:
\begin{description}
\item[(a)] When $t\rightarrow\infty$ we have, for any $n$, $\theta^{G_n}(t)=O(e^{-c_1t})$ for some $c_1>0$. We also have $\theta^{\infty}(t)=O(e^{-c_2t})$ for some $c_2>0$.
\item[(b)] When $t\rightarrow 0^+$ we have $\theta^{G_n}(t)= \Big(\prod_{i=1}^{d}a_i(n)\Big)e^{-2dt}I_0(2t)^d+O(t^{\min a_i(n)})$. We also have $\theta^{\infty}(t)= \frac{\prod_{i=1}^{d}\alpha_i}{(4\pi t)^\frac{d}{2}}+O(e^{-c_3/t})$ for some $c_3>0$.
\end{description}
\end{lemma}
\begin{proof}
The assertions in point (a) follow from the definition of $\theta^{\infty}$, Proposition \ref{theta} and our hypotheses on $\lambda_i(n)$ and $\lambda_i$. The first assertion in point (b) follow from the definition of $\theta^{G_n}$, together with the following estimate:
\begin{align*}
\Big|\theta_i^{G_n}(t)-a_i(n)e^{-2t}I_0(2t)\Big|&= \Big|a_i(n)e^{-2t}\sum_{k\neq 0}I_{ka_i(n)}(2t)\prod_{j=0}^{a_i(n)-1}w_{i,j}(n)^{-k}\Big|\\&\leq 2a_i(n)e^{-2t}\sum_{k\geq 1}I_{ka_i(n)}\\&\leq 2a_i(n)e^{-2t}I_0(2t)\frac{t^{a_i(n)}}{1-t^{a_i(n)}},
\end{align*}
where we used the bound $I_{ka_i(n)}(2t)=t^{ka_i(n)}\sum_{j\geq 0}\frac{t^{2j}}{j!(j+ka_i(n))!}\leq I_0(2t)t^{ka_i(n)}$.\\ The second assertion is a corollary of Lemma \ref{poisson}.
\end{proof}
\end{section}
\begin{section}{Asymptotics of $\det\Delta$}
In this section we establish an asymptotic formula for $\log\det\Delta$ when $n\longrightarrow\infty$, where $\Delta$ is the bundle Laplacian on $G_n$. We follow the steps of \cite{CJK10}. First we notice that, in our setting here, and in view of Remark \ref{eigen}, zero is not an eigenvalue. We begin with the following exact result.
\begin{theorem}
Let $$c_d:=-\int_0^\infty \Big(e^{-2dt}I_0(2t)^d-e^{-t}\Big)\frac{dt}{t}$$ and $$\calH_{d,n}:=-\int_0^\infty\Big(\theta^{G_n}(t)-\Big(\prod_{i=1}^{d}a_i(n)\Big)e^{-2dt}I_0(2t)^d\Big)\frac{dt}{t}.$$ Then $$\log(\det\Delta)=\Big(\prod_{i=1}^{d}a_i(n)\Big)c_d+\calH_{d,n}.$$
\end{theorem}
\begin{proof}
Thanks to the asymptotics derived in Lemma \ref{asymp}, we can proceed exactly in the same fashion as in section 3 of \cite{CJK10}. The only difference is that here $0$ is not an eigenvalue, so we do not need to substract $1$ in $\calH_{d,n}$, whence a slightly different expression for $\calH_{d,n}$.
\end{proof}
Now we need to understand the behavior of $\calH_{d,n}$ when $n\longrightarrow\infty$. First we observe that the discrete theta function $\theta^{G_n}$ converges to the continuous one $\theta^\infty$ when suitably normalized.
\begin{proposition}\label{rescaledtheta}
For all $t>0$ $$\lim\limits_{n\rightarrow\infty}\theta^{G_n}(n^2t)=\theta^\infty(t).$$
\end{proposition}
\begin{proof}
Our hypotheses on the weights imply that $$\lim\limits_{n\rightarrow\infty}\prod_{j=0}^{a_i(n)-1}w_{i,j}(n)^{-k}=e^{-2\pi ik\lambda_i}$$ for every $i\in\{1,\ldots,d\}$. This, together with Proposition 4.7 in \cite{CJK10}, leads to $$\lim\limits_{n\rightarrow\infty}a_i(n)e^{-2n^2t}I_{ka_i(n)}(2n^2t)\prod_{j=0}^{a_i(n)-1}w_{i,j}(n)^{-k}=\frac{\alpha_i}{\sqrt{4\pi t}}e^{-\frac{(\alpha_i k)^2}{4t}-2\pi ik\lambda_i}.$$
Since our weights have modulus one the bounds used in the proof of Proposition 5.2 in \cite{CJK10} are valid and allow us to exchange the limit and the infinite sum to deduce that $$\lim\limits_{n\rightarrow\infty}\theta_i^{G_n}(n^2t)=\theta^\infty(t),$$ by Lemma \ref{poisson}. We conclude the proof by taking the $d$-fold product.
\end{proof}
\begin{lemma}\label{expbound}
There exists a constant $c>0$ and an integer $n_0$ such that $$\theta^{G_n}(n^2t)\leq e^{-ct}$$ for any $t>0$ and $n\geq n_0$.
\end{lemma}
\begin{proof}
This is an adaptation of the proof of Lemma 5.3 in \cite{CJK10}.\\
Let $i\in\{1,\ldots,d\}$. If $\lambda_i\neq 1$ let $\epsilon_i$ be a real number such that $\epsilon_i>1$ and $\epsilon_i\lambda_i<1$. If $\lambda_i=1$ define $\epsilon_i=\frac{3}{2}$. Finally choose $n_0$ such that $\frac{a_i(n)}{n}\leq 2\alpha_i$ and $\frac{\lambda_i}{2}\leq \lambda_i(n)\leq\epsilon_i\lambda_i$ for every $n\geq n_0$.\\ For $0\leq j\leq a_i(n)-1$ such that $\frac{j+\lambda_i(n)}{a_i(n)}\leq \frac{1}{2}$, the elementary estimate $\sin(x)\geq x-\frac{x^3}{6}$ (valid for $0\leq x\leq \frac{\pi}{2}$) yields
\begin{align*}
a_i(n)\sin\Big(\frac{\pi(j+\lambda_i(n))}{a_i(n)}\Big)&\geq \pi(j+\lambda_i(n))-\frac{\pi^3}{6}(j+\lambda_i(n))^3\frac{1}{a_i(n)^2}\\&\geq \pi(j+\lambda_i(n))-\frac{\pi^3}{24}(j+\lambda_i(n))\\&\geq \pi (j+\frac{\lambda_i}{2}) (1-\frac{\pi^2}{24})\geq 0.
\end{align*}
Hence $$4n^2\sin^2\Big(\frac{\pi(j+\lambda_i(n))}{a_i(n)}\Big)\geq \frac{1}{\alpha_i^2}\Big(\pi (j+\frac{\lambda_i}{2}) (1-\frac{\pi^2}{24})\Big)^2$$ and we have $$\sum_{\substack{j=0\\2(j+\lambda_i(n))\leq a_i(n)}}^{a_i(n)-1}e^{-4tn^2\sin^2\Big(\frac{\pi(j+\lambda_i(n))}{a_i(n)}\Big)}\leq \sum_{j=0}^{\infty}e^{-c_1(j+\frac{\lambda_i}{2})^2t},$$ with $c_1=\Big(\frac{\pi}{\alpha_i}(1-\frac{\pi^2}{24})\Big)^2>0$.
For the other half of the sum defining $\theta_i^{G_n}(n^2t)$ we use the symmetry of the sine to write 
\begin{align*}
4n^2\sin^2\Big(\frac{\pi(j+\lambda_i(n))}{a_i(n)}\Big)&=4n^2\sin^2\Big(\frac{\pi(a_i(n)-j-\lambda_i(n))}{a_i(n)}\Big)\\&\geq \frac{1}{\alpha_i^2}\Big(\pi (a_i(n)-j-\lambda_i(n)) (1-\frac{\pi^2}{24})\Big)^2
\end{align*}
which leads to 
\begin{align*}
\sum_{\substack{j=0\\2(j+\lambda_i(n))> a_i(n)}}^{a_i(n)-1}e^{-4tn^2\sin^2\Big(\frac{\pi(j+\lambda_i(n))}{a_i(n)}\Big)}&\leq \sum_{\substack{j=0\\2(j+\lambda_i(n))> a_i(n)}}^{a_i(n)-1}e^{-c_1(a_i(n)-j-\lambda_i(n))^2t}\\&=\sum_{\substack{j=1\\2(j-\lambda_i(n))<a_i(n)}}^{a_i(n)}e^{-c_1(j-\lambda_i(n))^2t}\\&\leq e^{-c_1(1-\lambda_i(n))^2t}+\sum_{j\geq 2}e^{-c_1(j-\epsilon_i\lambda_i)^2t}.
\end{align*} 
The last expression is less than $1+\sum_{j\geq 2}e^{-c_1(j-\epsilon_i\lambda_i)^2t}$ if $\lambda_i=1$. But if $\lambda_i\neq 1$ it is less than $e^{-c_1(1-\epsilon_i\lambda_i)^2t}+\sum_{j\geq 2}e^{-c_1(j-\epsilon_i\lambda_i)^2t}$.
\\
Since there is at least one $i\in\{1,\ldots,d\}$ such that $\lambda_i\notin\{0,1\}$, at least one of the $\theta_i^{G_n}(n^2t)$ will have have exponential decay thanks to the bounds we just derived. The desired result then follows by taking the $d$-fold product.
\end{proof}
\begin{definition}
For $\Re(s)>\frac{d}{2}$, we define the \emph{Epstein-Hurwitz zeta function} as $$\zeta_{EH}(s;\alpha_1,\ldots,\alpha_d;\lambda_1,\ldots,\lambda_d)=\frac{1}{\Gamma(s)}\int_0^\infty\theta^\infty(t)t^{s-1}dt.$$
The integral is convergent by Lemma \ref{asymp}.
\end{definition}
We can define a more general version of the Epstein-Hurwitz zeta function, see \cite{Eli94} or \cite{Ber71}. This a particular case where the quadratic form is diagonal.
This function is a generalization of both the Epstein zeta function and the Hurwitz zeta function, whence its name. It can be seen as the spectral zeta function of the continuous "twisted" torus $\bbR^d/A\bbZ^d$ (where $A$ is the diagonal matrix with $\alpha_i$ on the diagonal), in the sense that functions on this torus are functions $u$ on $\bbR^d$ which are almost periodic, that is for all $x\in\bbR^d$: $$e^{-2\pi i \lambda_i}u(x+(0,\ldots,0,\alpha_i,0,\ldots,0))=u(x),$$ for every $1\leq i\leq d$.\\
We can also write it in the following, more familiar way (using the definition of $\theta^\infty$):
$$\zeta_{EH}(s;\alpha_1,\ldots,\alpha_d;\lambda_1,\ldots,\lambda_d)=(2\pi)^{-2s}\sum_{K\in\bbZ^d}\Big(\Big(\frac{k_1+\lambda_1}{\alpha_1}\Big)^2+\ldots+\Big(\frac{k_d+\lambda_d}{\alpha_d}\Big)^2\Big)^{-s}.$$
\\

Thanks to the asymptotic behavior of $\theta^\infty$ (see Lemma \ref{asymp}) we can compute the analytic continuation of $\zeta_{EH}$ by writing $$\int_0^\infty\theta^\infty(t)t^{s-1}dt=\int_1^\infty\theta^\infty(t)t^{s-1}dt+\int_0^1\Big(\theta^\infty(t)-\frac{\prod_{i=1}^{d}\alpha_i}{(4\pi t)^\frac{d}{2}}\Big)t^{s-1}dt+\frac{\prod_{i=1}^{d}\alpha_i}{(4\pi)^{\frac{d}{2}}(s-\frac{d}{2})},$$
where both integrals on the right-hand side define entire functions of $s$. This expression then provides a meromorphic continuation for $\zeta_{EH}$ to $\bbC$ with a simple pole at $s=\frac{d}{2}$. Note that it also implies that $\zeta_{EH}(-n;\alpha_1,\ldots,\alpha_d;\lambda_1,\ldots,\lambda_d)=0$ for all integers $n\geq 0$. It is then possible to find an expression for the derivative at $s=0$, using the fact that $\frac{1}{\Gamma(s)}=s+O(s^2)$ when $s\longrightarrow 0^+$:
\begin{align}\label{derzero}
\zeta_{EH}'(0;\alpha_1,\ldots,\alpha_d;\lambda_1,\ldots,\lambda_d)&=\int_1^\infty\theta^\infty(t)\frac{dt}{t}+\int_0^1\Big(\theta^\infty(t)-\Big(\prod_{i=1}^{d}\alpha_i\Big)(4\pi t)^{-\frac{d}{2}}\Big)\frac{dt}{t}\nonumber\\&-\frac{2}{d}\Big(\prod_{i=1}^{d}\alpha_i\Big)(4\pi)^{-\frac{d}{2}}.
\end{align}
\begin{proposition}
We have $$\lim\limits_{n\rightarrow\infty}\calH_{d,n}=-\zeta'_{EH}(0;\alpha_1,\ldots,\alpha_d;\lambda_1,\ldots,\lambda_d).$$
\end{proposition}
\begin{proof}
We split the integral in the definition of $\calH_{d,n}$ (after changing variables) in the following way:
\begin{align*}
\calH_{d,n}&=\int_0^\infty\Big(\theta^{G_n}(n^2t)-\Big(\prod_{i=1}^{d}a_i(n)\Big)e^{-2dn^2t}I_0(2n^2t)^d\Big)\frac{dt}{t}\\&=\int_1^\infty\theta^{G_n}(n^2t)\frac{dt}{t}-\Big(\prod_{i=1}^{d}a_i(n)\Big)\int_1^\infty e^{-2dn^2t}I_0(2n^2t)^d\frac{dt}{t}\\&+\int_0^1\Big(\theta^{G_n}(n^2t)-\Big(\prod_{i=1}^{d}a_i(n)\Big)e^{-2dn^2t}I_0(2n^2t)^d\Big)\frac{dt}{t}.
\end{align*}
Thanks to the bound obtained in Lemma \ref{expbound} we can change the limit with the integration sign in the first integral, which then converges to $$\int_1^\infty\theta^\infty(t)\frac{dt}{t},$$ by Proposition \ref{rescaledtheta}.\\
The second term converges to $$\frac{2}{d}\Big(\prod_{i=1}^{d}\alpha_i\Big)(4\pi)^{-\frac{d}{2}},$$
as proved in \cite{CJK10}.\\
The third integral converges to $$\int_0^1\Big(\theta^\infty(t)-\Big(\prod_{i=1}^{d}\alpha_i\Big)(4\pi t)^{-\frac{d}{2}}\Big)\frac{dt}{t},$$
using again the same result of \cite{CJK10} (the bounds used in their Proposition 5.5 can be used here thanks to the fact that our weights have modulus one). We conclude by using (\ref{derzero}).
\end{proof}
Thus we have proved that $$\log\det\Delta=\Big(\prod_{i=1}^{d}a_i(n)\Big)\calI_d-\zeta_{EH}'(0;\alpha_1,\ldots,\alpha_d;\lambda_1,\ldots,\lambda_d)+o(1),$$ which is Theorem \ref{thm11}.\\

This should be compared with the main theorem in \cite{CJK10}. In particular we see that the bundle has an influence on the second term only, the leading term being independent of the weights. In our opinion this formula has several interesting aspects. First it has a combinatorial interpretation in that, as explained previously, the determinant of the bundle Laplacian counts (with weights) the number of cycle-rooted spanning forests. Second it contains geometric information by relating the determinant of the bundle Laplacian on a line bundle over discrete weighted tori on the one hand and over a continuous torus on the other. Third it may have some number theoretic value, due to the Kronecker-type formula in Theorem \ref{Kronecker}. Finally it seems that physicists are also interested in quantities like $\zeta_{EH}'(0;\alpha_1,\ldots,\alpha_d;\lambda_1,\ldots,\lambda_d)$, see \cite{Eli94}.
\begin{example}\label{exdim1}
\emph{If $d=1$ the graph is a cycle and there is exactly one cycle-rooted spanning forest. It is then elementary to compute $\det\Delta$ using Theorem \ref{Kenyon}. We obtain $$\det\Delta=4\sin^2(\pi\lambda(n))=4\sin^2(\pi\lambda)+o(1),$$ when $n\longrightarrow\infty$.\\
On the other hand, we have $$\zeta_{EH}(s;\alpha;\lambda)=\sum_{k\in\bbZ}\frac{1}{\Big(\frac{k+\lambda}{\alpha}\Big)^{2s}}=\alpha^{2s}(\zeta(2s,\lambda)+\zeta(2s,1-\lambda)),$$ where we write $\zeta(s,\lambda)$ for the standard Hurwitz zeta function. Using the formulas $\zeta(0,a)=\frac{1}{2}-a$, $\zeta'(0,a)=\log\Gamma(a)-\frac{1}{2}\log(2\pi)$ and $\Gamma(z)\Gamma(1-z)=\frac{\pi}{\sin(\pi z)}$, we see that $$\zeta_{EH}'(0;\alpha;\lambda)=-2(\log\sin(\pi\lambda)+\log(2)).$$ Since $c_1=0$ (see \cite{CJK10}) this small computation confirms Theorem \ref{thm11} in dimension one.\\ Note that going in the opposite direction, this computation together with Theorem \ref{thm11} consitutes a proof of the reflection formula for the gamma function $$\Gamma(z)\Gamma(1-z)=\frac{\pi}{\sin(\pi z)}.$$}
\end{example}

When the dimension is $d=2$ there is a nice formula for the derivative of the Epstein-Hurwitz zeta function at $s=0$, stated in Theorem \ref{Kronecker}. It is very similar to the Kronecker limit formula, which has important applications in number theory, see for example the paper by Chowla and Selberg \cite{CS67}. The classical formula (which corresponds to the case with no bundle) involves the Dedekind eta-function, which is ubiquitous in the theory of modular form. The infinite product in Theorem \ref{Kronecker} can be considered as a generalization of the latter. As far as we know, the Epstein-Hurwitz zeta function has received little attention in the literature, with the exception of the papers \cite{Eli94} and \cite{Ber71}. The expression in Theorem \ref{Kronecker} does not appear explicitly in \cite{Ber71} and the formula proposed in \cite{Eli94} does not make apparent the analogy with the classical Kronecker limit formula.
\begin{proof}[Theorem \ref{Kronecker}]
In this proof we write $\zeta_{EH}(s)$ for $\zeta_{EH}(s;\alpha_1,\alpha_2;\lambda_1,\lambda_2)$ to simplify the notations.\\
First we note that the infinite product on the right-hand side is always positive, since we assumed that $\lambda_i\notin\{ 0,1\}$ for some $i$, so that the expression on the right-hand side is well-defined.\\
We can use Theorem 2 in the paper \cite{Ber71} by Berndt to write the Epstein-Hurwitz zeta function as an infinite sum of modified Bessel functions. There are three different cases that we have to treat separately: $\lambda_1\notin \{0,1\}$ and $\lambda_2\notin \{0,1\}$, $\lambda_1\in\{0,1\}$ and $\lambda_2\notin \{0,1\}$, $\lambda_1\notin \{0,1\}$ and $\lambda_2\in\{0,1\}$. We will explain the computations for the first case, the other ones being similar. So suppose $\lambda_1,\lambda_2\notin \{0,1\}$. Then, by Theorem 2 in \cite{Ber71} we have
\begin{align*}
(4\pi^3\alpha_1\alpha_2)^{-s}\Gamma(s)\zeta_{EH}(s)&=\Big(\frac{\alpha_1}{\alpha_2}\Big)^{1-s}\pi^{\frac{1}{2}-s}\Gamma(s-1/2)\Big(\zeta(2s-1,\lambda_2)+\zeta(2s-1,1-\lambda_2)\Big)\\&+2\sqrt{\frac{\alpha_1}{\alpha_2}}\sum_{m,n \atop m\neq 0}e^{-2\pi im\lambda_1}\Big|\frac{m}{n+\lambda_2}\Big|^{s-\frac{1}{2}}K_{s-\frac{1}{2}}\Big(2\pi\frac{\alpha_1}{\alpha_2}|m||n+\lambda_2|\Big),
\end{align*}
where $K_{s-\frac{1}{2}}$ is a modified Bessel function.\\
Alternatively, we can also start with the paper by Terras \cite{Ter73} and adapt the computations to our function $\zeta_{EH}$ to obtain the same representation in terms of Bessel functions.\\
Then we develop around $s=0$ using the fact that $\frac{1}{\Gamma(s)}=s+O(s^2)$ and the identity $K_{-\frac{1}{2}}(z)=\frac{1}{\sqrt{2}}\sqrt{\frac{\pi}{z}}e^{-z}$. Since the coefficient of the linear term in $s$ is the derivative at $s=0$, this leads to
\begin{align*}
\zeta_{EH}'(0)&=\frac{\alpha_1}{\alpha_2}\sqrt{\pi}\Gamma(-1/2)\Big(\zeta(-1,\lambda_2)+\zeta(-1,1-\lambda_2)\Big)+\sum_{m,n \atop m\neq 0}\frac{e^{-2\pi i m\lambda_1}}{|m|}e^{-2\pi\frac{\alpha_1}{\alpha_2}|m||n+\lambda_2|}\\&=2\pi\frac{\alpha_1}{\alpha_2}B_2(\lambda_2)+\sum_{n\in\bbZ}\sum_{m\geq 1}\Big(\frac{e^{2\pi i m\lambda_1}}{m}e^{-2\pi\frac{\alpha_1}{\alpha_2}m|n+\lambda_2|}+\overline{\frac{e^{2\pi i m\lambda_1}}{m}e^{-2\pi\frac{\alpha_1}{\alpha_2}m|n+\lambda_2|}}\Big)\\&=2\pi\frac{\alpha_1}{\alpha_2}B_2(\lambda_2)-\sum_{n\in\bbZ}\Big(\log\Big(1-e^{2\pi i\lambda_1-2\pi\frac{\alpha_1}{\alpha_2}|n+\lambda_2|}\Big)+\log\Big(\overline{1-e^{2\pi i\lambda_1-2\pi\frac{\alpha_1}{\alpha_2}|n+\lambda_2|}}\Big)\Big)\\&=2\pi\frac{\alpha_1}{\alpha_2}B_2(\lambda_2)-2\log\prod_{n\in\bbZ}\Big|1-e^{2\pi i\lambda_1}e^{-2\pi\frac{\alpha_1}{\alpha_2}|n+\lambda_2|}\Big|,
\end{align*}
where we used the special value $\zeta(-1,\lambda)=-\frac{B_{2}(\lambda)}{2}$, with $B_{2}(\lambda)=\lambda^2-\lambda+\frac{1}{6}$ the second Bernoulli polynomial. 
\end{proof}
With the same kind of computation we could in fact write a more general Kronecker-type formula for Epstein-Hurwitz zeta functions having non-diagonal quadratic form, see \cite{Ter73} and \cite{Ber71}.\\
An amusing consequence of Theorem \ref{Kronecker} is the following.
\begin{corollary}
The following identity is true: $$\frac{\displaystyle\prod_{n\geq 1}(1+e^{-2n\pi})}{\displaystyle\prod_{n\geq 0}(1-e^{-(2n+1)\pi})}=\frac{e^{\pi/8}}{\sqrt{2}}.$$
\end{corollary}
\begin{proof}
If $\alpha_1=\alpha_2$ the function $\zeta_{EH}$ is symmetric in $\lambda_1$ and $\lambda_2$ by definition. By Theorem \ref{Kronecker} this implies that $$2\pi B_2(\lambda_1)-2\log\prod_{n\in\bbZ}\Big|1-e^{2\pi i\lambda_2}e^{-2\pi|n+\lambda_1|}\Big|=2\pi B_2(\lambda_2)-2\log\prod_{n\in\bbZ}\Big|1-e^{2\pi i\lambda_1}e^{-2\pi|n+\lambda_2|}\Big|.$$
Taking $\lambda_1=0$ and $\lambda_2=\frac{1}{2}$ yields the result.
\end{proof}
Obviously we could write a whole family of similar identities, using other values for $\lambda_1$ and $\lambda_2$. It is likely that these formulas, or at least some of them, can be derived from the theory of Jacobi theta functions.\\

An interesting (and simple) choice of bundle is the following. We choose $d$ positive integers $m_1,\ldots,m_d$ and define, for each $i\in\{1,\ldots,d\}$, $w_{i,j}=1$ for all $0\leq j\leq m_i-2$ and $w_{i,m_i-1}=e^{2\pi i \lambda_i}=:z_i$. In words, we consider the $d$-fold cartesian product of the cyclic graphs $\bbZ/m_i\bbZ$ (as explained in Section 2), where, in each cycle, all the edges have trivial weight $1$ except for one edge which have weight $z_i=e^{2\pi i\lambda_i}$. One can think of this graph as a discrete $d$-dimensional torus constructed as follows: start with a $d$-dimensional cubic grid of size $m_1\times\ldots\times m_d$ with all edges having weight $1$ and add edges linking opposite boundaries, according to toric boundary conditions. For each pair of opposite boundaries, the corresponding edges all have weight $e^{2\pi i\lambda_i}$. For this example only, we allow all the weights to be trivial (we will come back to our earlier convention in Section 5).\\
Write $$F_{(m_1,\ldots,m_d)}(z_1,\ldots,z_d):=\det\Delta$$ for the determinant of the bundle Laplacian on the graph we just defined, if there exists $i$ such that $z_i\neq 1$. If $z_i=1$ for all $i$ write $$F_{(m_1,\ldots,m_d)}(z_1,\ldots,z_d):={\det}^*\Delta$$ for the product of the non-zero eigenvalues of the standard Laplacian.
We record the following easy result.
\begin{proposition}\label{prodform}
Let $m_1,\ldots,m_d$ and $n$ be positive integers.\\ For any choice of complex numbers $z_1,\ldots,z_d$ of modulus one, we have \begin{equation*}
F_{(m_1n,\ldots,m_dn)}(z_1,\ldots,z_d)= \prod_{u_1^{m_1}=z_1}\cdots\prod_{u_d^{m_d}=z_d}F_{(n,\ldots,n)}(u_1,\dots,u_d).\end{equation*}
\end{proposition}
\begin{proof}
Suppose first that not all $z_i$ are equal to $1$. Since the $m_i$-th roots of $z_i$ are given by $e^{2\pi i\frac{k+\lambda_i}{m_i}}$ for $0\leq k \leq m_i-1$ and in view of Remark \ref{eigen}, the logarithm of the right-hand side is equal to
$$
\sum_{k_1=0}^{m_1-1}\ldots\sum_{k_d=0}^{m_d-1}\sum_{j_1=0}^{n-1}\ldots\sum_{j_d=0}^{n-1}\log\Big(4\sin^2\Big(\frac{\pi}{n}(j_1+\\\frac{k_1+\lambda_1}{m_1}\Big)+\ldots+4\sin^2\Big(\frac{\pi}{n}(j_d+\frac{k_d+\lambda_d}{m_d})\Big)\Big),
$$
which can be rewritten as $$\sum_{j_1=0}^{m_1n-1}\ldots\sum_{j_d=0}^{m_dn-1}\log\Big(4\sin^2\Big(\frac{\pi}{m_1n}(j_1+\lambda_1)\Big)+\ldots+4\sin^2\Big(\frac{\pi}{m_dn}(j_d+\lambda_d)\Big)\Big),$$ which is $\log(LHS)$.\\
If $z_i=1$ for all $i$ almost the same computation works, taking into account the slightly different meaning of $F$ in that case.
\end{proof}
For instance, if we take $d=2$, $m_1=m_2=2$ and $z_1=z_2=1$ we get $$F_{(2n,2n)}(1,1)=F_{(n,n)}(1,1)F_{(n,n)}(1,-1)F_{(n,n)}(-1,1)F_{(n,n)}(-1,-1).$$
Since in that situation all cycles have monodromy $1$ or $-1$, Theorem \ref{Kenyon} tells us that $F_{(n,n)}(1,-1)$, $F_{(n,n)}(-1,1)$ and $F_{(n,n)}(-1,-1)$ are all integers multiple of $4$. We deduce that the number of spanning trees in the $2n\times 2n$ discrete torus is an integer multiple of the number of spanning trees in the $n\times n$ discrete torus (and the multiplicative constant is itself a multiple of $4$ determined by the cycle-rooted spanning forests in the $n\times n$ torus).\\
In dimension $2$, a very similar formula holds for the characteristic polynomial of the dimer model of any toroidal graph, see for example \cite{KOS06}. It would be interesting to investigate the case of other graphs. For which graphs do such a product formula for the bundle Laplacian holds?
\end{section}

\begin{section}{Asymptotics of zeta functions}
Now we give an asymptotic result about the spectral zeta function of $G_n$, in the same spirit as in \cite{FK16}.
\begin{definition}
For $\Re(s)>0$, the \emph{spectral zeta function} associated to the graph $G_n$ with bundle defined above is defined by $$\zeta_{G_n}(s)=\frac{1}{\Gamma(s)}\int_0^\infty\theta^{G_n}(t)t^{s-1}dt.$$
\end{definition}
In view of the asymptotics of the integrand obtained in Lemma \ref{asymp}, this integral is convergent in the domain specified in the definition. Note that, in order to simplify the notation, we do not write the dependence on the various parameters introduced earlier. More precisely, $\theta^{G_n}$ depends on the dimension $d$, the integers $a_i(n)$ and the weights $w_{i,j}$ and so does $\zeta_{G_n}$.
\\

In fact $\zeta_{G_n}$ is entire since Proposition \ref{theta} implies that \begin{equation*}
\zeta_{G_n}(s)=\frac{1}{4^s}\sum_K\frac{1}{\Big(\sin^2\Big(\frac{\pi(k_1+\lambda_1)}{a_1(n)}\Big)+\ldots+\sin^2\Big(\frac{\pi(k_d+\lambda_d)}{a_d(n)}\Big)\Big)^s},\end{equation*} where the index $K$ runs over all vectors $(k_1,\ldots,k_d)$, with $0\leq k_j\leq a_j(n)-1$.
\\

The spectral zeta function of $\bbZ^d$ is given by $$\zeta_{\bbZ^d}(s)=\frac{1}{\Gamma(s)}\int_0^\infty e^{-2dt}I_0(2t)^dt^{s-1}dt$$ for $0<\Re(s)<\frac{d}{2}$ and admits a meromorphic continuation with simple poles at $s=m+\frac{d}{2}$ ($m\geq 0$), see \cite{FK16}.\\

We are now ready to prove Theorem \ref{zeta}.
\begin{proof}[Theorem \ref{zeta}]
The proof is practically the same as in \cite{FK16}, using results from \cite{CJK10}. Recall that $$\theta^{G_n}(n^2t)\longrightarrow\theta^\infty(t)$$ when $n\longrightarrow\infty$ for all $t>0$, by Proposition \ref{rescaledtheta}. Next we write, for $0<Re(s)<\frac{d}{2}$, \begin{equation}\label{S1S2}
\zeta_{G_n}(s)=\frac{n^{2s}}{\Gamma(s)}\int_0^\infty\theta^{G_n}(n^2t)t^{s-1}dt=\frac{n^{2s}}{\Gamma(s)}\Big(S_1(n)+S_2(n)+S_3(n)+S_4(n)+S_5(n)\Big),\end{equation} where $$S_1(n):=\int_1^\infty\theta^{G_n}(n^2t)t^{s-1}dt,$$ $$S_2(n):=\int_0^1\Big(\theta^{G_n}(n^2t)-\Big(\prod_{i=1}^{d}a_i(n)\Big)e^{-2dn^2t}I_0(2n^2t)^d\Big)t^{s-1}dt,$$ $$S_3(n):=n^{-2s}\Big(\prod_{i=1}^{d}a_i(n)\Big)\Gamma(s)\zeta_{\bbZ^d}(s),$$ $$S_4(n):=-n^{-2s}\Big(\prod_{i=1}^{d}a_i(n)\Big)\int_{n^2}^{\infty}\Big(e^{-2dt}I_0(2t)^d-(4\pi t)^{-\frac{d}{2}}\Big)t^{s-1}dt$$ and $$S_5(n)=-n^{-2s}\Big(\prod_{i=1}^{d}a_i(n)\Big)\int_{n^2}^{\infty}(4\pi t)^{-\frac{d}{2}}t^{s-1}dt=\Big(\prod_{i=1}^{d}a_i(n)\Big)\frac{n^{-d}}{(4\pi)^\frac{d}{2}(s-\frac{d}{2})}.$$
Note that the equality (\ref{S1S2}) is in fact valid for $-\min a_i(n)<\Re(s)<\frac{d}{2}+1$, due to the analytic continuation of $\zeta_{\bbZ^d}$, Lemma \ref{asymp} and the asymptotic behavior of the modified Bessel function. Letting $n$ go to infinity, $S_1(n)$, $S_2(n)$ and $S_5(n)$ combine to give the second term in the asymptotics, for all $s\neq\frac{d}{2}$ (and a term $o(n^{2s})$). The first term is $S_3(n)$ and $S_4(n)$ contributes to the error term, as can be seen from the asymptotic behavior of $I_0$, if $\Re(s)<\frac{d}{2}+1$ \emph{a priori}. But we can use more terms from the asymptotics of $I_0$ at infinity to split the integral in $S_4(n)$ further, so that the validity of the equality (\ref{S1S2}) and, consequently, of the final asymptotics, extends to any $s$ such that $s\neq m+\frac{d}{2}$.
\end{proof}
This is in principle a more general result, because the function $\zeta_{G_n}$ contains a lot of information about the Laplace spectrum of $G_n$. In particular, if only the error term would be a bit smaller, say for example $o(\frac{n^{2s}}{\log n})$, we would obtain Theorem \ref{thm11} as a simple consequence by taking the derivative at $s=0$ and using $\zeta_{EH}(0;\alpha_1,\ldots,\alpha_d;\lambda_1,\ldots,\lambda_d)=0$ as observed above.
Indeed, the derivative at $s=0$ of $\zeta_{G_n}$ is equal to $-\log\det\Delta$, because $\zeta_{G_n}(s)=\sum\frac{1}{\lambda_j^s}$ where the $\lambda_j$ run over all the eigenvalues of $\Delta$. Moreover, if we take the derivative at $s=0$ in the similar asymptotic formula obtained in \cite{FK16} and compare with the results in \cite{CJK10}, we see that $-\zeta_{\bbZ^d}'(0)=c_d$.
\\

\bibliographystyle{plain}
\bibliography{bib-bundle}

\noindent Fabien Friedli 

\noindent Section de mathématiques \\
Université de Genève

\noindent 2-4 Rue du Lièvre\\
Case Postale 64

\noindent 1211 Genève 4, Suisse 

\noindent e-mail: fabien.friedli@unige.ch
\end{section}
\end{document}